\documentclass{amsart}%
\usepackage{amsmath}
\usepackage{amsfonts}
\usepackage{amssymb}
\usepackage{graphicx}%
\usepackage{xcolor}
\usepackage{hyperref}
\newtheorem{lemma}{Lemma}[section]
\newtheorem{theorem}[lemma]{Theorem}
\newtheorem{remark}[lemma]{Remark}

\newtheorem{coro}[lemma]{Corollary}
\newtheorem{definition}[lemma]{Definition}
\newtheorem{example}[lemma]{Example}

\title{Global Attractors of Non-autonomous Lattice Dynamical Systems}

\author{David Cheban}
\address[D. Cheban]{State University of Moldova\\
Vladimir Andrunachievici Instiy=tute of Mathematics and Computer Science\\
Laboratory of Differential Equations\\
A. Mateevich Street 60\\
MD--2009 Chi\c{s}in\u{a}u, Moldova} \email[D.
Cheban]{david.ceban@usm.md, davidcheban@yahoo.com}
\author{Andrei Sultan}
\address[A. Sultan]{%
State University of Moldova\\Vladimir Andrunachievici Institute of
Mathematics and Computer Science\\ Laboratory of Differential Equations\\A. Mateevich Street 60\\
MD--2009 Chi\c{s}in\u{a}u, Moldova} \email[A.
Sultan]{andrew15sultan@gmail.com}

\date{\today}

\begin{document}

\begin{abstract}
The aim of this paper is studying the compact global attractors
for non-autonomous lattice dynamical systems of the form
$u_{i}'=\nu (u_{i-1}-2u_i+u_{i+1})-\lambda u_{i}+f(u_i)+f_{i}(t)\
(i\in \mathbb Z,\ \lambda >0)$. We prove their dissipativness,
asymptotic compactness and then the existence of compact global
attractors.
\end{abstract}

\maketitle

\section{Introduction}\label{Sec1}

Denote by $\mathbb R :=(-\infty,\infty)$, $\mathbb Z :=\{0,\pm
1,\pm 2,\ldots\}$ and $\ell_{2}$ the Hilbert space of all
two-sided sequences $\xi =(\xi_{i})_{i\in \mathbb Z}$ ($\xi_{i}\in
\mathbb R$) with
\begin{equation}\label{eqI_1}
\sum\limits_{i\in \mathbb Z}|\xi_{i}|^{2}<+\infty \nonumber
\end{equation}
and equipped with the scalar product
\begin{equation}\label{eqI_2}
\langle \xi,\eta\rangle :=\sum\limits_{i\in \mathbb
Z}\xi_{i}\eta_{i} .\nonumber
\end{equation}
Let $(\mathfrak B, |\cdot|)$ be a Banach space with the norm
$|\cdot|$, $C(\mathbb R,\mathfrak B)$ be the space of all
continuous functions $f:\mathbb R\to \mathfrak B$ equipped with
the distance
\begin{equation}\label{eqI_3}
d(f_1,f_2):=\sup\limits_{L>0}\min\{\max\limits_{|t|\le
L}|f_{1}(t)-f_{2}(t)|,L^{-1}\}.
\end{equation}
The metric space $(C(\mathbb R,\mathfrak B),d)$ is complete and
the distance $d$, defined by (\ref{eqI_3}), generates on the space
$C(\mathbb R,\mathfrak B)$ the compact-open topology.

Let $h\in \mathbb R$, $f\in C(\mathbb R,\mathfrak B)$,
$f^{h}(t):=f(t+h)$ for any $t\in \mathbb R$ and $\sigma :\mathbb
R\times C(\mathbb R,\mathfrak B)\to C(\mathbb R,\mathfrak B)$ be a
mapping defined by $\sigma(h,f):=f^{h}$ for any $(h,f)\in \mathbb
R\times C(\mathbb R,\mathfrak B)$. Then \cite[Ch.I]{Che_2015} the
triplet $(C(\mathbb R,\mathfrak B),\mathbb R,\sigma)$ is a shift
dynamical system (or Bebutov's dynamical system) on he space
$C(\mathbb R,\mathfrak B)$. By $H(f)$ the closure in the space
$C(\mathbb R,\mathfrak B)$ of $\{f^{h}|\ h\in \mathbb R\}$ is
denoted.

In this paper we study the compact global attractors of the
systems
\begin{equation}\label{eqI1}
u_{i}'=\nu (u_{i-1}-2u_i+u_{i+1})-\lambda u_{i}+F(u_i)+f_{i}(t)\
(i\in \mathbb Z),
\end{equation}
where $\lambda >0$, $F\in C(\mathbb R, \mathbb R)$ and $f\in
C(\mathbb R,\ell_{2})$ ($f(t):=(f_{i}(t))_{i\in \mathbb Z}$ for
any $t\in \mathbb R$).

The system (\ref{eqI1}) can be considered as a discrete (see, for
example, \cite{BLW_2001}, \cite{HK_2023} and the bibliography
therein) analogue of a reaction-diffusion equation in $\mathbb R$:
\begin{equation}\label{eqI1.1}
 \frac{\partial{u}}{\partial{t}} = D\frac{\partial^{2}{u}}{\partial^{2}{x}}-\lambda u + F(u) +
 f(t,x),\nonumber
\end{equation}
where grid points are spaced $h$ distance apart and $\nu =
D/h^{2}$.

This study continues the first author's works devoted to the study
of compact global attractors of non-autonomous dynamical systems
\cite{Che_2015} and compact attractors of lattice dynamical
systems \cite{BLW_2001} (autonomous systems) and compact pullback
attractors \cite{HK_2023} (for non-autonomous systems).

The paper is organized as follows. In the second section we show
that under some conditions the equation (\ref{eqI1}) generates a
cocycle which plays a very important role in the study of the
asymptotic properties of the equation (\ref{eqI1}). In the third
section we prove that under some conditions the existence of an
absorbing set for the equation (\ref{eqI1}). The fourth section is
dedicated to the study the asymptotically compactness of the
cocycle generated by the equation (\ref{eqI1}). In the fifth
section we study the problem of existence of a compact global
attractor for the equation (\ref{eqI1}).

\section{Cocycles}\label{Sec2}

Consider a non-autonomous system
\begin{equation}\label{eq2.1}
u_{i}'=\nu (u_{i-1}-2u_i+u_{i+1})-\lambda u_{i}+F(u_i)+f_{i}(t)\
(i\in \mathbb Z) .
\end{equation}

Below we use the following conditions.

\emph{Condition }(\textbf{C1}). \label{C1} The function $f\in
C(\mathbb R,\mathfrak B)$ and it is translation-compact, i.e., the
set $\{f^{h}|\ h\in \mathbb R\}$ is pre-compact in the space
$C(\mathbb R,\mathfrak B)$.

\begin{lemma}\label{l2.1} \cite[Ch.IV,p.236]{Bro79},\cite[Ch.III]{Sel_1971},\cite[Ch.IV]{sib} The following statements are equivalent:
\begin{enumerate}
\item the function $f\in C(\mathbb R,\mathfrak B)$ is
translation-compact; \item the set $Q:=\overline{f(\mathbb R)}$ is
compact in $\mathfrak B$ and the function $f\in C(\mathbb
R,\mathfrak B)$ is uniformly continuous.
\end{enumerate}
\end{lemma}

\emph{Condition} (\textbf{C2}). \label{C2} The function $F\in
C(\mathbb R,\mathbb R)$ is Lipschitz continuous on bounded sets
and $F(0)=0$.

Denote by $ \widetilde{F}:\ell_{2}\to \ell_{2}$ the Nemytskii
operator generated by $F$, i.e.,
$\widetilde{F}(\xi)_{i}:=F(\xi_{i})$ for any $i\in \mathfrak N$.

\emph{Condition} (\textbf{C3}). \label{C3} $sF(s)\leq -\alpha s^2$
for any $s\in \mathbb R$.

\begin{definition}\label{defL1.8} A function $F\in C(Y\times \mathfrak B,\mathfrak
B)$ is said to be globally Lipschitzian (respectively locally
Lipschitzian) with respect to variable $u\in \mathfrak B$
uniformly with respect to $y\in Y$ if there exists a positive
constant $L$ (for any bounded set $B \subset \mathfrak B$ there
exists a constant $L_{B}$) such that
\begin{equation}\label{eqL1.81}
|F(y,u_1)-F(y,u_2)|\le L |u_1-u_2| \nonumber
\end{equation}
(respectively,
\begin{equation}\label{eqL1.82}
|F(y,v_1)-F(y,v_2)|\le L_{B} |v_1-v_2|) \nonumber
\end{equation}
for any $u_1,u_2\in \mathfrak B$ and $y\in Y$ (respectively $v_1,
v_2 \in B \subset \mathfrak B$ and $y\in Y$).
\end{definition}

\begin{definition}\label{defL2.8}
The smallest constant $L$ (respectively $L_{B}$) with the property
(\ref{eqL1.81}) is called Lipshchitz constant of function $F$
(notation $Lip(F)$, respectively $Lip_{B}(F)$).
\end{definition}

Let $B \subset \mathfrak B$, denote by $CL(Y\times B,\mathfrak B)$
the Banach space of any Lipschitzian functions $F\in C(Y\times
B,\mathfrak B)$ equipped with the norm
\begin{equation}
||F||_{CL}:=\max\limits_{y\in Y}|F(y,0)|+Lip_{B}(F).\nonumber
\end{equation}

\begin{lemma}\label{l2.2} \cite{BLW_2001} Under the Condition (\textbf{C2})
it is well defined the mapping $\widetilde{F}:\ell_{2}\to
\ell_{2}$ and
\begin{equation}\label{eq2.2}
\|\widetilde{F}(\xi)-\widetilde{F}(\eta)\|\le Lip_{B}(F)\|\xi
-\eta \| \nonumber
\end{equation}
for any $\xi,\eta\in \ell_{2}$, where $\|\cdot\|^{2}:=\langle
\cdot,\cdot \rangle$ and $\|\cdot \|$ is the norm on the space
$\ell_{2}$.
\end{lemma}

For any $u = (u_{i})_{i\in \mathbb Z}$, the discrete Laplace
operator $\Lambda$ is defined \cite[Ch.III]{HK_2023} from
$\ell_{2}$ to $\ell_{2}$ component wise by $\Lambda(u)_{i} =
u_{i-1} - 2u_{i} + u_{i+1}$ ($i\in \mathbb Z$). Define the
bounded linear operators $D^{+}$ and $D^{-}$ from $\ell_{2}$ to
$\ell_{2}$ by $(D^{+}u)_{i} = u_{i+1} - u_{i},\ (D^{-}u)_{i} =
u_{i-1} - u_{i}\ (i\in \mathbb Z)$.

Note that $\Lambda = D^{+}D^{-} = D^{-}D^{+}$ and $\langle D^{-}u,
v\rangle = \langle u, D^{+}v\rangle $ for any $u,v\in \ell_{2}$
and, consequently, $\langle \Lambda u,u \rangle = -|D^{+}u|^{2}\le
0$. Since $\Lambda$ is a bounded linear operator acting on the
space $\ell_{2}$, it generates a uniformly continuous semi-group
on $\ell_{2}$.

Under the Conditions (\textbf{C1}) and (\textbf{C2}) the system of
differential equations (\ref{eq2.1}) can be written in the form of
an ordinary differential equation
\begin{equation}\label{eq2.3}
u'=\nu \Lambda u +\Phi (u)+f(t)
\end{equation}
in the Banach space $\mathfrak B=\ell_{2}$, where
$\Phi(u):=-\lambda u +\widetilde{F} (u)$ and
$\Lambda(u)_{i}:=u_{i-1}-2u_{i}+u_{i+1}$ for any $u=(u_i)_{i\in
\mathbb Z}\in \ell_{2}$. Along with equation (\ref{eq2.3}) we
consider also it $H$-class, i.e., the family equations
\begin{equation}\label{eq2.3g}
u'=\nu \Lambda u +\Phi (u)+g(t),
\end{equation}
where $g\in H(f)$.

The family of equations (\ref{eq2.3g}) can be rewritten as follows
\begin{equation}\label{eq2.3H}
u'=F(\sigma(t,g),u)\ \ (g\in H(f)),
\end{equation}
where $F:H(f)\times \ell_{2}\to \ell_{2}$ is defined by
$F(g,u):=\nu \Lambda u+\Phi (u) +g(0)$. It easy to see that
$F(\sigma(t,g),u)=\nu \Lambda u+\Phi (u)+g(t)$ for any $(t,u,g)\in
\mathbb R\times \mathfrak B \times H(f)$.



Let $Y$ be a complete metric space, $(Y,\mathbb R,\sigma)$ be a
dynamical system on $Y$ and $ \Lambda $ be some complete metric
space of linear closed operators acting into Banach space $
\mathfrak B $. Consider the following linear differential equation
\begin{equation}\label{eqLS01.81}
x'=A(\sigma(t,y))x,\  \ (y\in Y)
\end{equation}
where $A\in C(Y,\Lambda)$. We assume that the following conditions
are fulfilled for equation (\ref{eqLS01.81}):
\begin{enumerate}
\item[a.] for any $ u \in \mathfrak B $ and $y\in Y $ equation
(\ref{eqLS01.81}) has exactly one solution that is defined on $
\mathbb R_{+} $ and satisfies the condition $ \varphi (0,u,y) = u
;$ \item[b.] the mapping $ \varphi : (t,u,y) \to \varphi (t,u,y) $
is continuous in the topology of $ \mathbb R_{+} \times \mathfrak
B \times Y$.
\end{enumerate}

Denote by $U(t,y):=\varphi(t,\cdot,y)$ for any $(t,y)\in \mathbb
R_{+}\times Y$.

Consider an evolutionary differential equation
\begin{equation}\label{eqSL1}
u'=A(\sigma(t,y))u + F(\sigma(t,y),u) \ \ (y\in Y)
\end{equation}
in the Banach space $\mathfrak B$, where $F$ is a nonlinear
continuous mapping ("small" perturbation) acting from $Y\times
\mathfrak B$ into $\mathfrak B$.

\begin{definition}
A function $u:[0,a)\mapsto \mathfrak B$ is said to be a weak
(mild) solution of equation (\ref{eqSL1}) passing through the
point $x\in \mathfrak B$ at the initial moment $t=0$ (notation
$\varphi(t,x,y)$) if $u\in C([0,T],\mathfrak B)$ and satisfies the
integral equation
\begin{equation}\label{eqSL3}
u(t)=U(t,y)x+\int_{0}^{t}U(t-s,\sigma(s,y))F(\sigma(s,y),u(s))ds\nonumber
\end{equation}
for any $t\in [0,T]$ and $0<T<a$.
\end{definition}

\begin{theorem}\label{thLS1.S1} \cite[Ch.VI]{Che_2020}
Suppose that the function $F\in C(Y\times \mathfrak B,\mathfrak
B)$ is locally Lipschitzian. Let $x_0\in \mathfrak B$, $r>0 $ and
the conditions listed above be fulfilled. Then, there exist
positive numbers $\delta =\delta (x_0,r)$ and $T=T(x_0,r) $ such
that equation (\ref{eqSL1}) admits a unique solution $\varphi
(t,x,y)$ ($x\in B[x_0,\delta]:=\{x\in \mathfrak B \ | \ \vert x
-x_0 \vert \le \delta \} $) defined on the interval $[0,T]$ with
the conditions: $\varphi (0,x,y)=x$, $\vert \varphi
(t,x,y)-x_0\vert \le r$ for any $t\in [0,T]$ and the mapping $
\varphi : [0,T]\times B[x_0,\delta] \times Y \to \mathfrak B\ (
(t,x,y)\mapsto \varphi (t,x,y))$ is continuous.
\end{theorem}

\begin{remark}\label{remSL1} Under the conditions of Theorem
\ref{thLS1.S1}:
\begin{enumerate}
\item if $\psi$ is a solution of equation (\ref{eqSL1}) on some
interval $[0,h]$, then $\psi$ can be extended over a maximal
interval of existence $[0,\alpha)$; \item if the solution $\psi$
is bounded, then $\psi$ can be extended on the interval
$[0,+\infty)$.
\end{enumerate}
\end{remark}

This statement can be proved using the same arguments as in the
case of ordinary differential equations (see, for example,
\cite[Ch.IV]{Bur}).

\begin{theorem}\label{th1.1} Under the Conditions (\textbf{C1}) and
(\textbf{C2}) there exist positive numbers $\delta =\delta
(u_0,r)$ and $T=T(u_0,r) $ such that equation (\ref{eqSL1}) admits
a unique solution $\varphi (t,g,y)$ ($u\in B[u_0,\delta]=\{u\in
\ell_{2}\ \|u -u_0 \| \le \delta \} $) defined on the interval
$[0,T]$ with the conditions: $\varphi (0,u,g)=u$, $\|\varphi
(t,u,g)-u_0\| \le r$ for any $t\in [0,T]$ and the mapping $
\varphi : [0,T]\times B[u_0,\delta] \times H(f) \to \ell_{2}\ (
(t,u,g)\mapsto \varphi (t,u,g))$ is continuous.
\end{theorem}
\begin{proof}
Assume that the Conditions (\textbf{C1}) and (\textbf{C2}) are
fulfilled. Consider the equation (\ref{eq2.3H}), where
$F(g,u):=\nu \Lambda u +\Phi(u)+g(0)$ for any $(u,g)\in
\ell_{2}\times H(f)$. It easy to check that under the conditions
of Theorem the mapping $F$ possesses the following properties:
\begin{enumerate}
\item $F$ is continuous; \item the mapping $F$ is locally
Lipschitzian in $u\in \ell_{2}$ uniformly with respect to $g\in
H(f)$, i.e., for any bounded subset $B\subset \ell_{2}$ there
exists a positive constant $L_{F}(B)$ such that
\begin{equation}\label{eqLB1}
\|F(u_1,g)-F(u_2,g)\|\le L_{F}(B)\|u_1-u_2\|\nonumber
\end{equation}
for any $u_1,u_2\in B$ and $g\in H(f)$; \item there exists a
positive constant $A$ such that
\begin{equation}\label{eqA1}
\|F(g,0)\|\le C\nonumber
\end{equation}
for any $g\in H(f)$.
\end{enumerate}

Now to finish the proof of Theorem it suffices to apply Theorem
\ref{thLS1.S1}.
\end{proof}

\begin{lemma}\label{lBC1} Assume that the conditions (\textbf{C1})--(\textbf{C3}) holds and $g \in H(f)$.
Then, for every $T > 0$, any solution $v(t)$ of the problem
(\ref{eq2.3g}) and $v(0)=v_0 \in \ell_{2}$ satisfies
\[
\|v(t)\| \le M \;,\quad \text{for all } 0 \le t \le T\,,
\]
where $M$ is a constant depending only on the data $(\lambda, C,
\|v_0\|)$ and $T$, where $C:=\sup\{\|f(t)\|\ | \ t\in \mathbb
R\}$.
\end{lemma}
\begin{proof} Let $v(t)$ be a solution of the equation
(\ref{eq2.3g}) with the initial condition $v(0)=v_0$ defined on
the maximal interval $[0,h)$. Denote by $y(t):=|v(t)|^{2}$ then we
have
\begin{eqnarray}\label{eqDE1}
& y'(t)=2(v'(t),v(t))=2(\nu \Lambda
v(t),v(t))+2(\Phi(v(t)),v(t))+2(g(t),v(t))= \nonumber \\
& -2\nu |D^{+}v(t)|^{2}-2\lambda
|v(t)|^{2}+2(\tilde{F}(v(t)),v(t))+2(g(t),v(t))
\end{eqnarray}
for any $t\in [0,h)$.

Since
\[
|(g(t),v(t))| \le \|g(t)\|\|v(t)\| \le \frac{1}{2}\lambda
\|v(t)\|^2 + \frac{1}{2\lambda}\|g(t)\|^2\,,
\]
using (\textbf{C3}) from (\ref{eqDE1}) we get
\begin{equation}\label{eqDE2}
y'(t)=2(v'(t),v(t)) \le -(\lambda +2\alpha)y(t)+
\frac{C^{2}}{\lambda}
\end{equation}
for any $t\in [0,h)$. By Gronwall's lemma from the inequality
(\ref{eqDE2}), taking into account that $y(0)=|v(0)|^{2}$, we
obtain
\begin{equation}\label{eqDE3}
y(t)\le e^{-(\lambda
+2\alpha)t}\big{(}(|v(0)|^{2}-\frac{C^{2}}{\lambda (\lambda
+2\alpha)})+\frac{C^{2}}{\lambda (\lambda
+2\alpha)}\big{)}\nonumber
\end{equation}
and, consequently,
\begin{equation}\label{eqDE4}
|v(t)|\le M \nonumber
\end{equation}
for any $0\le t\le T<h$, where
\begin{equation}\label{eqDE5}
M=M(T,|v_0|,C):=\Big{(}e^{-(\lambda
+2\alpha)T}\big{(}(|v_0|^{2}-\frac{C^{2}}{\lambda (\lambda
+2\alpha)})+\frac{C^{2}}{\lambda (\lambda
+2\alpha)}\big{)}\Big{)}^{1/2}.\nonumber
\end{equation}
Lemma is proved.
\end{proof}

\begin{remark}\label{remDC1} Lemma \ref{lBC1} remains true if we
replace the Condition (\textbf{C3}) by the weaker condition:
$F(s)s\le 0$ for any $s\in \mathbb R$.
\end{remark}

\begin{theorem} Under the Conditions (\textbf{C1})-(\textbf{C3}) the following statements hold:
\begin{enumerate}
    \item for any $(v,g)\in \ell_{2}\times H(f)$ there exists a unique
    solution $\varphi(t,v,g)$ of the equation (\ref{eq2.3g}) passing
    through the point $v$ at the initial moment $t=0$ and defined on
    the semi-axis $\mathbb R_{+}:=[0,+\infty)$; \item
    $\varphi(0,v,g)=v$ for any $(v,g)\in \ell_{2}\times H(f)$; \item
    $\varphi(t+\tau,v,g)=\varphi(t,\varphi(\tau,v,g),g^{\tau})$ for
    any $t,\tau\in \mathbb R_{+}$, $v\in \ell_{2}$ and $g\in H(f)$;
     \item the mapping
    $\varphi :\mathbb R_{+}\times \ell_{2}\times H(f)\to \ell_{2}$
    ($(t,v,g)\to \varphi(t,v,g))$ for any $(t,v,g)\in \mathbb
    R_{+}\times \ell_{2}\times H(f)$ is continuous.
\end{enumerate}
\end{theorem}
\begin{proof}
The first statement of Theorem follows from Lemma \ref{lBC1},
Theorem \ref{thLS1.S1} and Remark \ref{remSL1}.

The second and third statement are evident. The fourth statement
follows from Theorem \ref{thLS1.S1}.
\end{proof}

Let $Y$ be a complete metric space and $(Y,\mathbb R,\sigma)$ be a
dynamical system on $Y$.

\begin{definition}\label{defC1} Recall \cite[Ch.I]{Che_2015} that
$\langle \mathbb B,\varphi, (Y,\mathbb R,\sigma)\rangle$ is said
to be a cocycle over $(Y,\mathbb R,\sigma)$ with the fiber
$\mathfrak B$ if $\varphi$ is a continuous mapping acting from
$\mathbb R_{+}\times \mathfrak B\times Y\to \mathfrak B$
satisfying the following conditions:
\begin{enumerate}
\item $\varphi(0,u,y)=v$ for any $(v,y)\in \mathfrak B\times Y$;
\item
$\varphi(t+\tau,u,y)=\varphi(t,\varphi(\tau,u,t),\sigma(\tau,y))$
for any $(t,\tau \in \mathbb R_{+}$ and $(u,y)\in \mathfrak
B\times Y$.
\end{enumerate}
\end{definition}

\begin{coro}\label{corH1}
Under the conditions of Theorem \ref{th1.1} the equation
(\ref{eq2.3}) (respectively, the family of equations
(\ref{eq2.3g})) generates a cocycle $\langle
\ell_{2},\varphi,(H(f),\mathbb R,\sigma)\rangle$ over the shift
dynamical system $(H(f),\mathbb R,\sigma)$ with the fiber
$\ell_{2}$.
\end{coro}
\begin{proof} This statement directly follows from Theorem
\ref{th1.1} and Definition \ref{defC1}.
\end{proof}

\section{Existence of an absorbing set}\label{Sec3}

\begin{theorem}\label{th2.1} Under the Conditions (\textbf{C1})-(\textbf{C3})
there exists a closed ball $B[0,r]:=\{\xi \in \ell_{2}|\ |\xi|\le
r\}$ such that for any bounded subset $B\subset \ell_{2}$ there
exist a positive number $L=L(B)$ such that
$\varphi(t,B,Y)\subseteq B[0,r]$ for any $t\ge L(B)$, where
$\varphi(t,M,Y):=\{\varphi(t,u,y)|\ u\in M,\ y\in Y\}$.
\end{theorem}
\begin{proof}
Let $B=B[0,r]:=\{x\in \ell_{2}|\ ||x||\le r\}$ be a bounded
subset, $v\in B$ and $g\in H(f)$ then we have
\begin{eqnarray}\label{eqAS_1}
&\frac{d}{dt}\left\lVert \varphi(t,v,g)\right\rVert ^2 = 2\nu
\left\langle \Lambda \varphi(t,v,g), \varphi(t,v,g)\right\rangle
+\nonumber \\
& 2 \left\langle \Phi (\varphi(t,v,g)), \varphi(t,v,g)\right\rangle + 2 \left\langle g(t), \varphi(t,v,g)\right\rangle \nonumber \\
&\leq -2\lambda ||\varphi(t,v,g)||^{2}+ 2 \sum_{i \in \mathbb{Z} } \varphi_{i}(t,v,g) f(\varphi_{i}(t,v,g)) +
2 \sum_{i \in \mathbb{Z} } g_{i}(t)\varphi_{i}(t,v,g) \nonumber \\
&\leq -2\lambda ||\varphi(t,v,g)||^{2} -2\alpha \left\lVert \varphi(t,v,g)\right\rVert ^2 + \lambda||\varphi(t,v,g)||^{2}+ \frac{\left\lVert g(t)\right\rVert ^2}{\lambda} \notag \\
&\leq -(2\alpha +\lambda)\left\lVert \varphi(t,v,g)\right\rVert ^2
+ \frac{C^{2}}{\lambda}\nonumber
\end{eqnarray}
where the penultimate step follows from Young's inequality because
$||g(t)|| \leq C:=\sup\{||f(t)||:\ t\in \mathbb R\}$ for any $g
\in H(f)$). Hence, Gronwall's lemma implies that
\[
\left\lVert \varphi(t,v,g) \right\rVert^2 \leq \left\lVert v
\right\rVert^2 e^{-(2\alpha +\lambda)t} + \frac{C ^2}{\lambda
(\lambda +2\alpha)} \left( 1 - e^{-(2\alpha +\lambda) t} \right),
\quad t \ge 0.
\]
Define the closed ball \( Q \) in \( \ell^2 \) by
\[
Q := \left\{ u \in \ell^2 : \left\lVert u \right\rVert^2 \le R^2
:= 1 + \frac{C^2}{\lambda (\lambda +2\alpha)} \right\}
\]
then we have
\begin{equation}\label{eqAS1}
||\varphi(t,v,g)||^{2}\le R^{2}\nonumber
\end{equation}
for any $r>R$ and
\begin{equation}\label{eqAS2}
t\ge L(B):=\frac{1}{\lambda (\lambda +2\alpha)}\ln (r^{2}-R^{2}
+1).\nonumber
\end{equation}
Theorem is proved.
\end{proof}

\section{Asymptotically compactness of the cocycle generated by
the equation (\ref{eq2.3})}\label{Sec4}

Let $\langle \mathfrak B,\varphi, (Y,\mathbb R,\sigma)\rangle$ (or
shortly $\varphi$) be a cocycle over dynamical system $(Y,\mathbb
R,\sigma)$ with the compact phase space $Y$.

Let $A$ and $B$ be two bounded subsets from $\mathfrak B$. Denote
by $\rho(a,b):=|a-b|$ ($a,b\in \mathfrak B$),
$\rho(a,B):=\inf\limits_{b\in B}\rho(a,b)$ and
\begin{equation}\label{eqAC1_1}
\beta(A,B):=\sup\limits_{a\in A}\rho(a,B).\nonumber
\end{equation}

\begin{definition}\label{defAC1} A dynamical system $(X,\mathbb
R_{+},\pi)$ is called asymptotically compact \cite[Ch.I]{Che_2015}
if for any bounded, closed and positively invariant subset
$B\subset X$ there exists a nonempty compact subset $K=K(B)$ such
that
\begin{equation}\label{eqAC1_2}
\lim\limits_{t\to +\infty}\beta(\pi(t,B),K)=0.\nonumber
\end{equation}
\end{definition}

\begin{definition}\label{defAC2} A cocycle $\langle W,\varphi,(Y,\mathbb
T,\sigma)\rangle$is said to be asymptotically compact if the
skew-product dynamical system $(X,\mathbb R_{+},\pi)$ generated by
cocycle $\varphi$ ($X:=W\times Y$, $\pi :=(\varphi,\sigma)$) is
asymptotically compact.
\end{definition}



\begin{definition}\label{defCD1} A cocycle $\langle \ell_{2},\varphi,(Y,\mathbb
R,\sigma)\rangle$ satisfy an asymptotic tails property on the
bounded set $K\subset \ell_{2}$, if for arbitrary positive number
$\varepsilon$ there exist positive numbers $L(\varepsilon)$ and
$k_{\varepsilon}\in \mathbb N$ such that
\begin{equation}\label{eqAT1}
\sum\limits_{|k|\ge k_{\varepsilon}}
|\varphi_{k}(t,v,y)|^{2}<\varepsilon
\end{equation}
for any $t\ge L(\varepsilon)$ and $(v,y)\in K\times Y$.
\end{definition}

\begin{lemma}\label{lCD1} Assume that the skew product dynamical
system $\langle \ell_{2},\varphi,(Y,\mathbb T,\sigma)\rangle$
satisfies the following conditions:
\begin{enumerate}
\item the metric space $Y$ is compact; \item the cocycle $\varphi$
admits a bounded absorbing set $K\subset \ell_{2}$, i.e., for any
bounded subset $B\subset \ell_{2}$ there exists a positive number
$L=L(B)$ such that $\varphi(t,B,Y)$ $\subset $ $K$ for any $t\ge
L$ (or equivalently: $\mathcal K:=K\times Y$ is an absorbing set
for the skew-product dynamical system $(X,\mathbb R_{+},\pi)$
generated by cocycle $\varphi$); \item the cocycle $\varphi$
satisfy an asymptotic tails property on the absorbing set
$\mathcal K \subset X$.
\end{enumerate}

Then the skew-product dynamical system $(X,\mathbb R_{+},\pi)$
generated by the cocycle $\varphi$ is asymptotically compact.
\end{lemma}
\begin{proof} Let $B$ be a bounded, closed and positively invariant
subset of $X$. Since the space $Y$ is compact then there exists a
positive number $r>0$ such that $B\subseteq B[0,r]\times Y$.
Consider the sequences $\{x_n\}\subset B$ and $t_n\to +\infty$ as
$n\to \infty$. We will show that the sequence $\{\pi(t_n,x_n)\}$
is precompact in the space $X:=\ell_{2}\times Y$ or equivalently
the sequence $\{\varphi(t_n,u_n,y_n)\}$ is precompact in
$\ell_{2}$ ($x_n=(u_n,y_n)$ and
$\pi(t_n,x_n)=(\varphi(t_n,u_n,y_n),\sigma(t_n,y_n))$ for any
$n\in \mathbb N$) because $\{y_n\}\subset Y$ and the space $Y$ is
compact.

Since the est $B$ is positively invariant then $\pi(t,B)\subseteq
B\subseteq B[0,r]\times Y$ for any $t\ge 0$ and, consequently,
\begin{equation}\label{eqBK1}
\varphi(t,u,y)\in B[0,r] \nonumber
\end{equation}
for any $(u,y)\in B$ and $t\ge 0$. In particular, we have
\begin{equation}\label{BK2}
|\varphi(t_n,u_n,y_n)|\le r \nonumber
\end{equation}
for any $n\in \mathbb N$). Thus the sequence
$\{\varphi(t_n,u_n,y_n)\}$ is weakly precompact in $\ell_{2}$ and,
consequently, we can assume that $\{\varphi(t_n,u_n,y_n\}$ is
weakly convergent, i.e,
\begin{equation}\label{eqWC1}
\varphi(t_n,u_n,y_n)\rightharpoonup u \nonumber
\end{equation}
weakly in $\ell_{2}$.

Let $\varepsilon$ be an arbitrary positive number. Since $K\subset
\ell_{2}$ is an absorbing set for the cocycle $\varphi$ then there
exists a positive number $L_{1}=L_{1}(B[0,r])$ such that
\begin{equation}\label{eqAS_01}
\varphi(t,u,y)\in K
\end{equation}
for any $t\ge L_{1}(B[0,r])$ and $(u,y)\in B[0,r]\times Y$. Let
$n_1\in \mathbb N$ be a number such that $t_n\ge L_{1}$ for any
$n\ge n_1$ and, consequently, from (\ref{eqAS_01}) we obtain
\begin{equation}\label{eqAS_2}
\varphi(t_n,u_n,y_n)\in K
\end{equation}
for any $n\ge n_1$.

We will show that the sequence $\{\varphi(t_n,u_n,y_n)\}$ converge
in the space $\ell_{2}$, that is, for any $\varepsilon >0$ there
exists a number $n(\varepsilon)\in \mathbb N$ such that
\begin{equation}\label{eqAS3}
\|\varphi(t_n,u_n,y_n)-u\|<\varepsilon \nonumber
\end{equation}
for any $n\ge n(\varepsilon)$.

By (\ref{eqAS_2}) we have
\begin{equation}\label{eqAS4}
\varphi(L_{1},u_n,y_n)\in K
\end{equation}
for any $n\in \mathbb N$. Since the cocycle $\varphi$ satisfy an
asymptotic tails on the set $K$ for given $\varepsilon
>0$ there exist $k_{1}(\varepsilon)\in \mathbb N$ and
$L_{1}(\varepsilon)>0$ such that
\begin{equation}\label{eqAS5}
\sum\limits_{|i|\ge
k_{1}(\varepsilon)}|\varphi_{i}(t,\varphi(L_{1},u_n,y_n),\sigma(L_{1},y_n))|^{2}<\varepsilon^{2}/8
\end{equation}
for any $t\ge L_{1}(\varepsilon)$.

Since $t_n\to +\infty$ as $n\to \infty$, there exists
$n_{2}(\varepsilon)\in \mathbb N$ such  that if $n\ge
n_{2}(\varepsilon)$, then $t_n - L_{1}\ge L_{1}(\varepsilon)$, and
hence from (\ref{eqAS5}), we have
\begin{eqnarray}\label{eqAS6}
& \sum\limits_{|i|\ge
k_{1}(\varepsilon)}|\varphi_{i}(t_n,u_n,y_n)|^{2}=
\\
& \sum\limits_{|i|\ge
k_{1}(\varepsilon)}|\varphi_{i}(t_n-L_{1},\varphi(L_{1}u_n,y_n),\sigma(L_1,y_n))|^{2}\le
\varepsilon^{2}/8.\nonumber
\end{eqnarray}

Since $u\in \ell_{2}$ then there exists $k_{2}(\varepsilon)$ such
that
\begin{equation}\label{eqAS7}
\sum\limits_{|i|\ge k_{2}(\varepsilon)}|u_{i}|^{2}\le
\varepsilon^{2}/8.
\end{equation}
Let $k(\varepsilon)=\max\{k_1(\varepsilon),k_{2}(\varepsilon)\}$.
By the weak convergence of $\{\varphi(t_n,u_n,y_n)\}$ we have that
\begin{equation}\label{eqAS8_1}
\varphi_{i}(t_n,u_n,y_n)\to u_{i}\nonumber
\end{equation}
as $n\to \infty$ (for any $|i|\le k(\varepsilon)$), which implies
that there exists $n_{3}(\varepsilon)$ such that for any $n\ge
n_{3}(\varepsilon)$ we obtain
\begin{equation}\label{eqAS8}
\sum\limits_{|i|\le
k(\varepsilon)}|\varphi_{i}(t_n,u_n,y_n)-u_{i}|^{2}\le
\varepsilon^{2}/2
\end{equation}
for any $n\ge n_{3}(\varepsilon).$

Setting
$n(\varepsilon):=\max\{n_{1},n_{2}(\varepsilon),n_{3}(\varepsilon)\}$,
from (\ref{eqAS6})-(\ref{eqAS8}) we get that, for $n\ge
n(\varepsilon)$,
\begin{eqnarray}\label{eqAS9}
& \|\varphi(t_n,u_n,y_n)-u\|^{2}=\sum\limits_{|i|\le
k(\varepsilon)}|\varphi_{i}(t_n,u_n,y_n)-u_{i}|^{2} +\nonumber\\
& \sum\limits_{|i|>
k(\varepsilon)}|\varphi_{i}(t_n,u_n,y_n)-u_{i}|^{2}\le
\varepsilon^{2}/2+2\sum\limits_{|i|\ge
k_{\varepsilon}}(|\varphi_{i}(t_n,u_n,y_n)|^{2}+|u_{i}|^{2})\le
\varepsilon^{2}.\nonumber
\end{eqnarray}
as desired. Hence we obtain $\varphi(t_n,u_n,y_n)$ converges to $u$
in the space $\ell_{2}$. Lemma is completely proved.
\end{proof}

\begin{theorem}\label{thAC1} Under the Conditions
(\textbf{C1})-(\textbf{C3}) the cocycle $\langle
\ell_{2},\varphi,(H(f),\mathbb R,\sigma)\rangle$ generated by the
equation (\ref{eq2.3}) satisfy an asymptotic tails property on the
set $Q$.
\end{theorem}
\begin{proof}
We will prove this statement using the ideas and methods
elaborated in the work \cite{BLW_2001} (see also
\cite[Ch.2.3]{HK_2023}). Consider a smooth function \( \xi :
\mathbb{R}^+ \to [0, 1] \) satisfying
\[
\xi(s) =
\begin{cases}
0, & 0 \le s \le 1, \\
\in [0, 1], & 1 \le s \le 2, \\
1, & s \ge 2
\end{cases}
\]
and note that there exists a constant $C_0$ such that
$|\xi'(s)| \le C_0$ for all \( s \ge 0 \). Then for a fixed \( k
\in \mathbb{N} \) (its value will be specified later), define
\[
\xi_k(s) = \xi \left( \frac{s}{k} \right) \quad \text{for all} \quad s \in \mathbb{R_+}.
\]

Given \( u \in C^{1}(\mathbb R_{+},\ell_{2}) \), define \( v \in
C^{1}(\mathbb R_{+},\ell_{2}) \) componentwise as
\[
v_i (t):= \xi_k( |i| ) u_i(t) \quad \text{for} \quad i \in
\mathbb{Z}\ \ (\forall \ t\in \mathbb R_{+}).
\]

Note that
\begin{equation}\label{eqNT1}
\langle u(t),v(t)\rangle =\sum\limits_{i\in \mathbb
Z}\xi_{k}(|i|)|u_{i}(t)|^{2}
\end{equation}
and
\begin{equation}\label{eqNT2}
\frac{d\langle u(t),v(t)\rangle}{dt}=2\langle
\frac{du(t)}{dt},v(t)\rangle
\end{equation}
for any $t\in \mathbb R_{+}$.

Taking the inner product of equation (\ref{eq2.3g}) with \(
\mathbf{v}(t) \) gives
\[
\frac{d}{dt} \langle \mathbf{u}(t), \mathbf{v}(t) \rangle + \nu
\langle \mathrm{D}^+ \mathbf{u}(t), \mathrm{D}^+ \mathbf{v}(t)
\rangle = \langle \Phi(\mathbf{u}(t)), \mathbf{v}(t) \rangle +
\langle \mathbf{g(t)}, \mathbf{v}(t) \rangle,
\]

that is
\begin{equation}\label{20}
\frac{d}{dt} \sum_{i \in \mathbb{Z}} \xi_k( |i| ) |u_i|^2 + 2 \nu
\langle \mathrm{D}^+ \mathbf{u}, \mathrm{D}^+ \mathbf{v} \rangle =
2 \sum_{i \in \mathbb{Z}} \xi_k( |i| ) u_i f(u_i) + 2 \sum_{i \in
\mathbb{Z}} \xi_k( |i| ) g_{i}(t)u_i
\end{equation}

Each term in the equation (\ref{20}) will now be estimated. First,
\[
\left\langle \mathrm{D}^+ \mathbf{u}, \mathrm{D}^+ \mathbf{v} \right\rangle = \sum_{i \in \mathbb{Z}} (u_{i+1} - u_i)(v_{i+1} - v_i)
\]
\[
= \sum_{i \in \mathbb{Z}} (u_{i+1} - u_i) \left[ \left( \xi_k(|i+1|) - \xi_k(|i|) \right) u_{i+1} + \xi_k(|i|)(u_{i+1} - u_i) \right]
\]
\[
= \sum_{i \in \mathbb{Z}} \left( \xi_k(|i+1|) - \xi_k(|i|) \right) (u_{i+1} - u_i) u_{i+1} + \sum_{i \in \mathbb{Z}} \xi_k(|i|) (u_{i+1} - u_i)^2
\]
\[
\ge \sum_{i \in \mathbb{Z}} \left( \xi_k(|i+1|) - \xi_k(|i|) \right) (u_{i+1} - u_i) u_{i+1}.
\]
Since
\[
\left| \sum_{i \in \mathbb{Z}} \left( \xi_k(|i+1|) - \xi_k(|i|) \right) (u_{i+1} - u_i) u_{i+1} \right|
\le \sum_{i \in \mathbb{Z}} \frac{1}{k} |\xi'(s_i)| \cdot |u_{i+1} - u_i| \cdot |u_{i+1}|,
\]
for some \( s_i \) between \( |i| \) and \( |i+1| \), and
\[
\sum_{i \in \mathbb{Z}} |\xi'(s_i)| \, |u_{i+1} - u_i| \, |u_{i+1}|
\le C_0 \sum_{i \in \mathbb{Z}} \left( |u_{i+1}|^2 + |u_i||u_{i+1}| \right) \le 4 C_0 \left\| \mathbf{u} \right\|^2.
\]

Then it follows that for all \( \mathbf{u} \in Q \) and \( \mathbf{v} \in \ell^2 \) defined componentwise as \( v_i := \xi_k(|i|) u_i \), for \( i \in \mathbb{Z} \),
\begin{equation}\label{In20}
\left\langle \mathrm{D}^+ \mathbf{u}, \mathrm{D}^+ \mathbf{v} \right\rangle
\ge - \frac{4 C_0 \|Q\|^2}{k}.
\end{equation}
where \( \|Q\| := \sup_{\mathbf{u} \in Q} \left\| \mathbf{u}
\right\| \). On the other hand, by Condition (\textbf{C3}),
\[
2 \sum_{i \in \mathbb{Z}} \xi_k(|i|) u_i f(u_i) \le -2 \alpha \sum_{i \in \mathbb{Z}} \xi_k(|i|) |u_i|^2
\]

and by Young's inequality
\[
2 \sum_{i \in \mathbb{Z}} \xi_k(|i|) g_i u_i \le
\alpha \sum_{i \in \mathbb{Z}} \xi_k(|i|) |u_i|^2 +
\frac{1}{\alpha} \sum_{i \in \mathbb{Z}} \xi_k(|i|) |g_i|^2.
\]

Thus
\begin{equation}\label{22}
2 \sum_{i \in \mathbb{Z}} \xi_k(|i|) u_i f(u_i) + 2 \sum_{i \in
\mathbb{Z}} \xi_k(|i|) g_i(t) u_i \le -\alpha \sum_{i \in
\mathbb{Z}} \xi_k(|i|) |u_i|^2 + \frac{1}{\alpha} \sum_{|i| \ge k}
|g_i|^2
\end{equation}

Using the estimates (\ref{In20}) and (\ref{22}) in the equation (\ref{20}) gives
\begin{equation}\label{EqCDN1}
\frac{d}{dt} \sum_{i \in \mathbb{Z}} \xi_k(|i|) |u_i|^2 + \alpha
\sum_{i \in \mathbb{Z}} \xi_k(|i|) |u_i|^2 \le \nu \frac{4 C_0 \|
Q \|^2}{k} + \frac{1}{\alpha} \sum_{|i| \ge k} |g_i|^2
\end{equation}

Since $g\in H(f)$ and the set $H(f)$ is a compact subset in the
space $C(\mathbb R,\ell_{2})$ then by Lemma \ref{l2.1} the set
$\overline{f(\mathbb R)}$ is a compact subset of $\ell_{2}$. In
particular for any $\varepsilon >0$ there exists a natural number
$k(\varepsilon)$ such that
\begin{equation}\label{eqDC2}
\sum\limits_{|i|\ge k(\varepsilon)}|v_i|^{2}<\varepsilon
\end{equation}
for any $v\in \overline{f(\mathbb R)}$. Note that $g\in H(f)$ and,
consequently,
\begin{equation}\label{eqDC3}
\overline{g(\mathbb R)}\subseteq \overline{f(\mathbb R)}.
\end{equation}
From (\ref{eqDC2}) and (\ref{eqDC3}) we obtain
\begin{equation}\label{eqDC4}
\sum\limits_{|i|\ge k(\varepsilon)}|g_i(t)|^{2}<\varepsilon
\nonumber
\end{equation}
for any $g\in H(f)$ and $t\in \mathbb R$.

Note that $g(t)\in \overline{g(\mathbb R)}\subseteq
\overline{f(\mathbb R)}$ and, consequently, for every \(
\varepsilon
> 0 \), there exists $k_{\varepsilon}$ such that
\[
\nu \frac{4C_0 \|Q\|^2}{k} + \frac{1}{\alpha} \sum_{|i| \ge k}
|g_{i}(t)|^2 \le \varepsilon, \quad k \ge k(\varepsilon)
\]
for any $g\in H(f)$ and $t\in \mathbb R$.

The inequality (\ref{EqCDN1}) along with the relation above give
\[
\frac{d}{dt} \sum_{i \in \mathbb{Z}} \xi_k(|i|) |u_i|^2 + \alpha
\sum_{i \in \mathbb{Z}} \xi_k(|i|) |u_i|^2 \le \varepsilon
\]
for any $g\in H(f)$ and $t\in \mathbb R$.

Then, Gronwall's lemma implies that
\[
\sum_{i \in \mathbb{Z}} \xi_k(|i|) |\varphi(t,v,g)_i(t,
\mathbf{u}_o)|^2 \le e^{-\alpha t} \sum_{i \in \mathbb{Z}}
\xi_k(|i|) |v_{i}|^2 + \frac{\varepsilon}{\alpha} \le e^{-\alpha
t} \|v\|^2 + \frac{\varepsilon}{\alpha}.
\]

Hence for every \( v \in Q \),
\[
\sum_{i \in \mathbb{Z}} \xi_k(|i|) |\varphi(t,v,g)_{i}|^2 \le
e^{-\alpha t} \| Q \|^2 + \frac{\varepsilon}{\alpha}.
\]
and therefore
\[
\sum_{i \in \mathbb{Z}} \xi_k(|i|) \left| u_i(t, \mathbf{u}_o) \right|^2 \le \frac{2\varepsilon}{\alpha}, \quad \text{for } t \ge T(\varepsilon) := \frac{1}{\alpha} \ln \frac{\alpha \| Q \|^2}{\varepsilon}.
\]
This means that the cocycle $\langle \ell_{2},\varphi,(Y,\mathbb
R,\sigma)\rangle$ is asymptotic tails on the absorbing set $Q$.
Theorem is proved.
\end{proof}

\section{Compact global attractors}\label{Sec5}

\begin{definition}\label{defCGA0_1} A family $\{I_{y}|\ y\in Y\}$ of
compact subsets $I_{y}$ of $\mathfrak B$ is said to be a compact
global attractor for the cocycle $\langle \mathfrak
B,\varphi,(Y,\mathbb R,\sigma)\rangle$ if the following conditions
are fulfilled:
\begin{enumerate}
\item the set
\begin{equation}\label{eqCGA1}
\mathcal I :=\bigcup \{I_{y}|\ y\in Y\}\nonumber
\end{equation}
is precompact; \item the family of subsets $\{I_{y}|\ y\in Y\}$ is
invariant, i.e., $\varphi(t,I_{y},y)=I_{\sigma(t,y)}$ for any
$(t,y)\in \mathbb R_{+}\times Y$; \item
\begin{equation}\label{eqCGA2}
\lim\limits_{t\to +\infty}\sup\limits_{y\in
Y}\beta(\varphi(t,M,y),\mathcal I)=0\nonumber
\end{equation}
\end{enumerate}
for any compact subset $M$ from $\mathfrak B$.
\end{definition}

\begin{definition}\label{defCGA1} A cocycle $\varphi$ is said to be
dissipative if there exists a bounded subset $K\subset \mathfrak
B$ such that for any bounded subset $B\subset \ \mathfrak B$ there
exists a positive number $L=L(B)$ such that
$\varphi(t,B,Y)\subseteq K$ for any $t\ge L(B)$, where
$\varphi(t,B,Y):=\{\varphi(t,u,y)|\ (u,y)\in B\times Y\}$.
\end{definition}

\begin{theorem}\label{thCGA1} \cite[Ch.II]{Che_2024} Assume that the metric space $Y$ is
compact and the cocycle $\langle \mathfrak B,\varphi,(Y,\mathbb
R,\sigma)\rangle$ is dissipative and asymptotically compact.

Then the cocycle $\varphi$ has a compact global attractor.
\end{theorem}

\begin{theorem}\label{thCGA2} Under the Conditions
(\textbf{C1})-(\textbf{C3}) the equation (\ref{eq2.3}) (the
cocycle $\varphi$ generated by the equation (\ref{eq2.3})) has a
compact global attractor $\{I_{g}|\ g\in H(f)\}$.
\end{theorem}
\begin{proof} This statement follows from Theorems \ref{th2.1},
\ref{thAC1} and \ref{thCGA1}.
\end{proof}

Below we give an example which illustrate our general results.

\begin{example}\label{exGA1} Let $\{omega_{i}\}_{i\in \mathbb Z}$ be a
sequence of real numbers. For every $i\in \mathbb Z$ we define a
function $f_{i}\in C(\mathbb R,\mathbb R)$ by the equality
\begin{equation}\label{eqGA1}
f_{i}(t):=\frac{\sin(\omega_{i}t+\ln(1+t^{2}))}{2^{|i|}} \nonumber
\end{equation}
for all $t\in \mathbb R$.

Note that the functions $f_{i}$ ($i\in \mathbb Z$) possesses the
following properties:
\begin{enumerate}
\item
\begin{equation}\label{eqGA_0}
|f_{i}(t)|\le \frac{1}{2^{|i|}}
\end{equation}
for all $t\in \mathbb R$ and $i\in \mathbb Z$; \item
\begin{equation}\label{eqGA_1}
|f_{i}'(t)|\le (2+|\omega_{i}|)
\end{equation}
for all $t\in \mathbb R$ and $i\in \mathbb Z$.
\end{enumerate}

\begin{lemma}\label{lGA_1} For every $i\in \mathbb Z$ the
function $f_{i}$ is bounded and uniformly continuous on $\mathbb
R$.
\end{lemma}
\begin{proof} This statement directly follows from (\ref{eqGA_0})
and (\ref{eqGA_1}).
\end{proof}

\begin{lemma}\label{lGA1} The following statement hold:
\begin{enumerate}
\item $f(t)\in \ell_{2}$, where $f(t):=(f_{i}(t))_{i\in \mathbb
Z}$; \item for every $\varepsilon >0$ there exists a number
$n(\varepsilon)\in \mathbb N$ such that
\begin{equation}\label{eqGA2}
\sum\limits_{|i|\ge
n(\varepsilon)}|f_{i}(t)|^{2}<\frac{\varepsilon^{2}}{4}
\end{equation}
for all $t\in \mathbb R$.
\end{enumerate}
\end{lemma}
\begin{proof} Since
\begin{equation}\label{eqGA3}
\|f(t)\|^{2}=\sum\limits_{i\in \mathbb
Z}|f_{i}(t)|^{2}=\sum\limits_{i\in \mathbb
Z}\frac{\sin^{2}(\omega_{i}t+\ln(1+t^{2}))}{2^{2|i|}}\le
\sum\limits_{i\in \mathbb Z}\frac{1}{2^{2|i|}}=\frac{11}{3}
\nonumber
\end{equation}
then $f(t)\in \ell_{2}$.

Note that for every $\varepsilon >0$ there exists a number
$n(\varepsilon)\in \mathbb N$ such that
\begin{equation}\label{eqGA4}
\sum\limits_{|i|\ge
n(\varepsilon)}\frac{1}{4^{|i|}}<\frac{\varepsilon^{2}}{4} .
\end{equation}

By (\ref{eqGA_0}) and (\ref{eqGA4}) we obtain
\begin{equation}\label{eqGA5}
\sum\limits_{|i|\ge n(\varepsilon)}|f_{i}(t)|^{2}\le
\sum\limits_{|i|\ge
n(\varepsilon)}\frac{1}{4^{|i|}}<\frac{\varepsilon^{2}}{4}
\end{equation}
for all $t\in \mathbb R$.
\end{proof}

Consider the function $f:\mathbb R\to \ell_{2}$ defined by
$f(t):=(f_{i}(t))_{i\in \mathbb Z}$ for all $t\in \mathbb R$.

\begin{lemma}\label{lGA2} The following statements hold:
\begin{enumerate}
\item the function $f:\mathbb R\to \ell_{2}$ is uniformly
continuous on $\mathbb R$; \item the set $f(\mathbb R)$ is a
precompact subset of $\ell_{2}$.
\end{enumerate}
\end{lemma}
\begin{proof}
For every $\varepsilon >0$ we choose $n(\varepsilon)\in \mathbb N$
such that (\ref{eqGA5}) holds. Since the functions $f_{i}$ ($|i|<
n(\varepsilon)$) are uniformly continuous then for $\varepsilon$
there exists a positive number $\delta =\delta(\varepsilon)$ such
that $|t_1-t_2|<\delta$ implies (see Lemma \ref{lGA_1})
\begin{equation}\label{eqGA6}
\sum\limits_{|i|<
n(\varepsilon)}|f_{i}(t_1)-f_{i}(t_2)|^{2}<\frac{\varepsilon^{2}}{2}
.
\end{equation}

On the other hand we have
\begin{eqnarray}\label{eqGA7}
\|f(t_1)-f(t_2)\|^{2}=\sum\limits_{|i|<
n(\varepsilon)}|f_{i}(t_1)-f_{i}(t_2)|^{2}+ \sum\limits_{|i|\ge
n(\varepsilon)}|f_{i}(t_1)-f_{i}(t_2)|^{2}
\end{eqnarray}
for any $t_1,t_2\in \mathbb R$. From (\ref{eqGA7}), (\ref{eqGA6})
and (\ref{eqGA2}) we receive
\begin{eqnarray}\label{eqGA8}
& \|f(t_1)-f(t_2)\|^{2}=\sum\limits_{|i|<
n(\varepsilon)}|f_{i}(t_1)-f_{i}(t_2)|^{2}+ \sum\limits_{|i|\ge
n(\varepsilon)}|f_{i}(t_1)-f_{i}(t_2)|^{2}\le \nonumber \\
& \sum\limits_{|i|< n(\varepsilon)}|f_{i}(t_1)-f_{i}(t_2)|^{2}+
\sum\limits_{|i|\ge
n(\varepsilon)}2(|f_{i}(t_1)|^{2}+|f_{i}(t_2)|^{2})\nonumber \\
&  < \frac{\varepsilon^{2}}{2}+
2(\frac{\varepsilon^{2}}{4}+\frac{\varepsilon^{2}}{4})=\varepsilon^{2}
\end{eqnarray}
and, consequently, $\|f(t_1)-f(t_2)\|<\varepsilon$ for any
$t_1,t_2\in \mathbb R$ with $|t_1-t_2|<\delta$.

Let now $v$ be an arbitrary element of the set $f(\mathbb R)$ then
there exists a number $s\in \mathbb R$ such that $v=f(s)$. By
Lemma \ref{lGA1} (item (ii)) for every $\varepsilon >0$ there
exists a number $n(\varepsilon)\in \mathbb N$ such that
\begin{equation}\label{eqGA9}
\sum\limits_{|i|\ge n(\varepsilon)}|v_i|^{2}= \sum\limits_{|i|\ge
n(\varepsilon)}\|f_i(s)|^{2}<\frac{\varepsilon^{2}}{4} \nonumber
\end{equation}
and, consequently, by Theorem 5.25 \cite[Ch.V,p.167]{LS_1974} the
subset $f(\mathbb R)$ of $\ell_{2}$ is precompact.
\end{proof}

\begin{coro}\label{corGA1} The function $f$ is Lagrange stable,
i.e., the set $H(f)$ is a compact subset of $C(\mathbb
R,\ell_{2})$.
\end{coro}
\begin{proof} This statement follows from Lemmas \ref{l2.1} and
\ref{lGA2}.
\end{proof}

Consider the system of differential equations
\begin{equation}\label{eqAP2}
u_{i}' = \nu (u_{i-1} - 2u_i + u_{i+1}) - \lambda u_{i} + F(u_i) +
\frac{\sin (\omega_{i} t +\ln(1+t^{2}))}{2^{|i|}} \quad (i \in
\mathbb{Z}),
\end{equation}
where $F(u) = -u - u^{3}$ for all $u \in \mathbb{R}$.

Along with this system of equations (\ref{eqAP2}), consider the
(equivalent) equation
\begin{equation}\label{eqAP3}
u' = \Lambda u - \lambda u + \tilde{F}(u) + f(t)
\end{equation}
in the space $\ell_{2}$.

Taking into account the results above it is easy to show that the
Conditions (\textbf{C1})-(\textbf{C3}) for the equation
(\ref{eqAP3}) are fulfilled. According to our result this equation
admits a compact global attractor.
\end{example}

\section{Funding}

This research was supported by the State Program of the Republic
of Moldova "Monotone Nonautonomous Dynamical Systems
(24.80012.5007.20SE)" and partially was supported by the
Institutional Research Program 011303 "SATGED", Moldova State
University.

\section{Conflict of Interest}

The authors declare that the have not conflict of interest.

\medskip
\textbf{ORCID (D. Cheban):} https://orcid.org/0000-0002-2309-3823

\textbf{ORCID (A. Sultan):} https://orcid.org/0009-0003-9785-9291

\end{document}